\documentclass[]{article}

\usepackage{graphicx}
\usepackage{amsmath}
\usepackage{amssymb}
\usepackage{amsthm}
\usepackage{pxfonts}
\usepackage{enumerate}
\usepackage{color}
\usepackage{mathdots}
\usepackage{sectsty}
\usepackage[hidelinks]{hyperref}
\usepackage{tikz}
\usepackage{caption}
\usepackage{adjustbox}

\sectionfont{\scshape\centering\fontsize{11}{14}\selectfont}
\subsectionfont{\scshape\fontsize{11}{14}\selectfont}
\usepackage{fancyhdr}

\newcommand\shorttitle{Moments of moments of characteristic polynomials of random unitary matrices and lattice point counts}
\newcommand\authors{Theodoros Assiotis and Jonathan P. Keating}

\newtheorem{thm}{Theorem}[section]

\newtheorem{defn}[thm]{Definition}
\newtheorem{rmk}[thm]{Remark}
\newtheorem{prop}[thm]{Proposition}

\title{\large \bf MOMENTS OF MOMENTS OF CHARACTERISTIC POLYNOMIALS OF RANDOM UNITARY MATRICES AND LATTICE POINT COUNTS}
\author{\small THEODOROS ASSIOTIS AND JONATHAN P. KEATING}
\date{}

\begin{document}

\maketitle

\begin{abstract}
In this note we give a combinatorial and non-computational proof of the asymptotics of the integer moments of the moments of the characteristic polynomials of Haar distributed unitary matrices as the size of the matrix goes to infinity. This is achieved by relating these quantities to a lattice point count problem. Our main result is a new explicit expression for the leading order coefficient in the asymptotic as a volume of a certain region involving continuous Gelfand-Tsetlin patterns with constraints.
\end{abstract}

\tableofcontents

\section{Introduction}

\subsection{Main result}

Let
\begin{align*}
P_N(U,\theta)=\det \left(I-Ue^{-\i \theta}\right)
\end{align*}
denote the characteristic polynomial on the circle (here and throughout $\i=\sqrt{-1}$) of a matrix $U \in \mathbb{U}(N)$, where $\mathbb{U}(N)$ is the group of $N\times N$ unitary matrices. 

We endow $\mathbb{U}(N)$ with the normalized Haar measure and denote by $\mathbb{E}$ the mathematical expectation with respect to it. We are interested in the following quantities, which we call the moments of the moments:
\begin{align}
\textnormal{MoM}_N\left(k,\beta\right)=\mathbb{E}\left[\left(\frac{1}{2\pi}\int_{0}^{2\pi}|P_N(U,\theta)|^{2\beta}d\theta\right)^k\right].
\end{align}

We give here a new proof, alternative to the one in \cite{BaileyKeating}, of a formula for the asymptotics of $\textnormal{MoM}_N\left(k,\beta\right)$ as $N \to \infty$ for integer $k,\beta$. Our main result is a new explicit expression for the constant in the leading order term in the asymptotic expansion, which we relate to a lattice point counting problem. Specifically, we prove the following theorem.

\begin{thm}\label{MainTheorem}
Let $k,\beta \in \mathbb{N}$. Then, 
\begin{align}
\textnormal{MoM}_N\left(k,\beta\right)=\mathfrak{c}(k,\beta)N^{k^2\beta^2-k+1}+O\left(N^{k^2\beta^2-k}\right)
\end{align}
where $\mathfrak{c}(k,\beta)$ can be written explicitly as a volume of a certain region involving continuous Gelfand-Tsetlin patterns with constraints, see Section 4 where we provide an explicit formula for it.
\end{thm}

\begin{rmk}\label{RemarkPoly}
$\textnormal{MoM}_N\left(k,\beta\right)$ is actually a polynomial in the variable $N$, see \cite{BaileyKeating}.
\end{rmk}

\begin{rmk}
It appears to be very hard to obtain the explicit expression for $\mathfrak{c}(k,\beta)$ given in this paper from the one in \cite{BaileyKeating}. In addition to its aesthetic appeal, having this new expression in terms of a volume is important due to connections to the theory of integrable systems and Painleve equations. In particular, for $k=2$ such an expression (obtained in \cite{KeatingRodgersRodittyGershonRudnick}) was used in \cite{BasorGeRubinstein} to relate $\mathfrak{c}(2,\beta)$ to the Painleve V transcendent, see Section 4 for more details.
\end{rmk}

\subsection{Historical overview}
Theorem \ref{MainTheorem} proves a conjecture of Fyodorov and Keating, for integer values of $k$ and $\beta$, on the asymptotics of $\textnormal{MoM}_N\left(k,\beta\right)$. It is also closely related to the following conjecture, see \cite{FyodorovKeating}, on the maximum of the logarithm of the absolute value of the characteristic polynomial $\log|P_N(U,\theta)|$, with $A \in \mathbb{U}(N)$ Haar distributed:
\begin{align*}
\underset{0\le \theta <2\pi}{\max}\log|P_N(U,\theta)|=\log N-\frac{3}{4}\log \log N+\mathfrak{x}_N(A),
\end{align*}
where $\mathfrak{x}_N(A)$ has a limiting distribution that is the sum of two Gumbel random variables. This is because, at least formally as $\beta \to \infty$, the quantities $\textnormal{MoM}_N\left(k,\beta\right)$ for real $k$ determine the distribution of $\underset{0\le \theta <2\pi}{\max}\log|P_N(U,\theta)|$.  There is expected to be a freezing transition at $k\beta^2=1$ which determines the large $\beta$ limit. However, this heuristic picture is far from being fully understood and it remains a major problem to make it rigorous. We refer to \cite{FyodorovKeating}, \cite{BaileyKeating} for further motivation and background on these conjectures.

The case $k=1$ of Theorem \ref{MainTheorem} is classical, see for example \cite{KeatingSnaith}, and $\textnormal{MoM}_N\left(1,\beta\right)$ can be calculated explicitly (in fact this can be done for any real $\beta$) using the celebrated Selberg integral (see \cite{Forrester}):
\begin{align*}
\textnormal{MoM}_N\left(1,\beta\right)=\prod_{1\le i,j \le \beta-1}^{}\left(1+\frac{N}{i+j+1}\right).
\end{align*}
In particular, we have the following expression for the leading order coefficient:
\begin{align*}
\mathfrak{c}(1,\beta)=\prod_{j=0}^{\beta-1}\frac{j!}{(j+\beta)!}.
\end{align*}

In the case $k=2$, the asymptotics of $\textnormal{MoM}_N\left(2,\beta\right)$ for real $\beta$ have been established by Claeys and Krasovsky in \cite{ClaeysKrasovsky} using Riemann-Hilbert problem techniques, in particular by proving a uniform asymptotic formula for Toeplitz determinants with two coalescing Fisher-Hartwig singularities. As a by-product of their approach they are able to obtain a representation for $\mathfrak{c}(2,\beta)$ in terms of the Painleve V transcendent.

For $k=2$ and $\beta \in \mathbb{N}$ Theorem \ref{MainTheorem} is also proven in \cite{KeatingRodgersRodittyGershonRudnick}. In that paper the authors provided two different proofs and also two different expressions for $\mathfrak{c}(2,\beta)$. One of the proofs is complex analytic in nature and makes use of a multiple contour integral expression for $\textnormal{MoM}_N\left(2,\beta\right)$. This approach was then extended in \cite{BaileyKeating} for general $k\in \mathbb{N}$, which obtains Theorem \ref{MainTheorem} with a different expression for $\mathfrak{c}(k,\beta)$, by performing a quite intricate asymptotic analysis of a multidimensional contour integral (in particular the case $k\ge 3$ is considerably harder than $k=2$).

The other proof in \cite{KeatingRodgersRodittyGershonRudnick} is combinatorial in nature and first relates the problem to counting Gelfand-Tsetlin patterns with certain constraints and then equivalently to counting lattice points in certain convex regions of Euclidean space, for which known results on their asymptotics can be used. It is this approach that we extend to general $k\in \mathbb{N}$.  This leads to a new expression for $\mathfrak{c}(k,\beta)$. Surprisingly, unlike in the complex analytic calculation, this new combinatorial approach does not present any significant additional technical difficulties for $k\ge 3$ compared to $k=2$ when it comes to determining the order of the polynomial referred to in Remark \ref{RemarkPoly}. This makes our proof more accessible. However, obtaining a simple formula for $\mathfrak{c}(k,\beta)$ remains a difficulty for $k\ge 3$.

Finally, as we already mentioned the formula coming from the combinatorial approach was used in \cite{BasorPainleveV} to obtain a Hankel determinant representation for $\mathfrak{c}(2,\beta)$.  This provides an alternative route to relating $\mathfrak{c}(2,\beta)$ to the Painleve V equation. We shall explain briefly this argument in Section 4.

\subsection{Organization of the paper}
In Section 2 we give three equivalent combinatorial representations for $\textnormal{MoM}_N\left(k,\beta\right)$: one in terms of semistandard Young tableaux, one in terms of Gelfand-Tsetlin patterns, and finally one in terms of integer arrays/lattice points. In Section 3, making use of a classical theorem on the asymptotics of the number of lattice points in convex regions of Euclidean space, we prove Theorem \ref{MainTheorem}. In Section 4, we study the leading order coefficient $\mathfrak{c}(k,\beta)$ and give an explicit expression for it in terms of volumes of continuous Gelfand-Tsetlin patterns/polytopes with constraints (see \cite{OlshanskiProjections}), which themselves have explicit expressions in terms of B-splines, well-known objects in approximation theory and total positivity, see \cite{CurrySchoenberg},\cite{Karlin}, \cite{OlshanskiProjections}. Finally, we explain, following \cite{BasorPainleveV}, how to relate $\mathfrak{c}(2,\beta)$ to the Painleve V transcendent.

\paragraph{Acknowledgements.} We are grateful to Emma Bailey for discussions.  The research described here was supported by ERC Advanced Grant  740900 (LogCorRM).  JPK also acknowledges support from a Royal Society Wolfson Research Merit Award.  

\section{Equivalent combinatorial representations}

\subsection{Preliminaries on Young tableaux and Gelfand-Tsetlin patterns}
In order to make this note self-contained, we first give the necessary background, following the exposition in \cite{GorinRahman}, on semistandard Young tableaux and Gelfand-Tsetlin patterns. We begin with a number of definitions.
 
\begin{defn}
A partition $\lambda=(\lambda_1,\lambda_2,\dots)$ is a finite sequence of non-negative integers such that $\lambda_1\ge \lambda_2 \ge \lambda_3\ge \dots \ge 0$. We write $|\lambda|=\sum_{i=1}^{\infty}\lambda_i<\infty$ for the size of the partition and $l(\lambda)$ for its length, namely the number of positive $\lambda_i$. We identify a partition with a Young diagram: a left-justified shape consisting of square boxes with $l(\lambda)$ rows with lengths (number of boxes) $\lambda_1,\dots, \lambda_{l(\lambda)}$. Given a pair of Young diagrams $\lambda, \nu$ such that $\lambda$ is contained in $\nu$ we write $\nu \setminus \lambda$ for the cells of $\nu$ that are not in $\lambda$.
\end{defn}

\begin{defn}
A non-negative signature $\lambda$ of length $M$ is a non-negative and non-increasing sequence of integers that has length $M$: $\lambda=(\lambda_1\ge \lambda_2 \ge \dots \ge \lambda_M \ge 0)$. We emphasize the importance of keeping track of the number of trailing zeroes. We will also still write $\lambda$ for the unique Young diagram corresponding to the signature $\lambda$. We denote the set of such signatures by $\mathsf{S}^+_M$. 
\end{defn}

\begin{rmk}
As the name suggests, signatures can involve negative integers as well. Here, our interest in them stems from the fact that we want to keep track of the number of trailing zeroes (which is a feature that distinguishes them from partitions). The terminology "non-negative signature" is taken from Section 4.3 of \cite{BorodinOlshanskiYoungBouquet}.
\end{rmk}

\begin{defn}
A semistandard Young tableau of shape $\lambda$, for some signature $\lambda \in \mathsf{S}^+_M$, is an insertion of the numbers $\{1,2, \dots, M\}$ into the cells of the corresponding Young diagram such that the entries weakly increase along each row and strictly increase along each column (see for example Section 2.1 in \cite{GorinRahman}). We denote the set of such tableaux by $\mathsf{SSYT}(\lambda)$.
\end{defn}

\begin{defn}
We say that two non-negative signatures $\lambda^{(M)}\in \mathsf{S}_M^+$ and $\lambda^{(M+1)}\in \mathsf{S}_{M+1}^+$ interlace and write $\lambda^{(M)} \prec \lambda^{(M+1)}$ if:
\begin{align*}
\lambda_1^{(M+1)}\ge \lambda_1^{(M)}\ge \lambda_2^{(M+1)}\ge \cdots \ge \lambda_{M}^{(M+1)}\ge \lambda_M^{(M)}\ge \lambda_{M+1}^{(M+1)}.
\end{align*}
\end{defn}

\begin{defn}
A non-negative Gelfand-Tsetlin pattern of length/depth $M$ is a sequence of signatures $\{\lambda^{(i)} \}_{i=1}^M$ such that $\lambda^{(i)} \in \mathsf{S}^+_i$ and:
\begin{align*}
\lambda^{(1)}\prec \lambda^{(2)}\prec \lambda^{(3)}\prec \cdots \prec \lambda^{(M-1)}\prec \lambda^{(M)}.
\end{align*}
We write $\mathsf{GT}_M^+$ for the set of all such patterns and $\mathsf{GT}_M^+(\nu)$ for the ones with fixed top signature $\nu \in \mathsf{S}_M^+$.
\end{defn}

Let $\nu \in \mathsf{S}_M^+$. It is well-known, see for example Section 2 in \cite{GorinRahman}, that there exists a bijection 
$\mathfrak{B}$ between $\mathsf{GT}_M^+(\nu)$ and $\mathsf{SSYT}(\nu)$. This is described as follows:
\begin{itemize}
\item Given $(\lambda^{(1)}\prec \lambda^{(2)} \prec \cdots \prec \lambda^{(M)}=\nu) \in \mathsf{GT}_M^+(\nu)$, the corresponding tableau $t \in \mathsf{SSYT}(\nu)$ is obtained by inserting value $i \le M$ into the cells of $\lambda^{(i)} \setminus \lambda^{(i-1)}$. In case this is empty then $i$ is not inserted.
\item Conversely, given $t \in \mathsf{SSYT}(\nu)$ the corresponding Gelfand-Tsetlin pattern  $(\lambda^{(1)}\prec \lambda^{(2)} \prec \cdots \prec \lambda^{(M)}=\nu) \in \mathsf{GT}_M^+(\nu)$ is obtained as follows. Set $\lambda^{(i)}$ to be the Young diagram consisting of the cells of $t$ with entries $\le i$ and removing trailing zeroes to ensure that $\lambda^{(i)}\in \mathsf{S}_i^+$.
\end{itemize}

\subsection{Semistandard Young tableaux and Gelfand-Tsetlin pattern representation}

Assume $k,\beta \in \mathbb{N}$. Our starting point is the following result, which is already implicit in the investigation in \cite{BaileyKeating}. We give a proof for completeness. 
\begin{prop}
Let $k,\beta \in \mathbb{N}$. $\textnormal{MoM}_N\left(k,\beta\right)$ is equal to the number of semistandard Young tableaux of shape $(N,\dots,N,0,\dots,0)$ where each of the strings of $N$'s and $0$'s is of length $k\beta$ and which moreover satisfy the extra condition that there have to be $N\beta$ entries from each of the sets
\begin{align}\label{SemistandardConstraint}
\{2 \beta(j-1)+1,\dots, 2j \beta \}, \ \textnormal{ for } j \in \{1, \dots, k\}.
\end{align}
\end{prop}

\begin{proof}
First, using Fubini's theorem we obtain:
\begin{align*}
\textnormal{MoM}_N\left(k,\beta\right)=\left(\frac{1}{2\pi}\right)^k\int_{0}^{2\pi}\cdots \int_{0}^{2\pi}\mathbb{E}\left(\prod_{j=1}^{k} |P_N\left(U,\theta_j\right)|^{2\beta}\right)d\theta_1\cdots d\theta_k.
\end{align*}
Now, recall that we have the following representation due to Bump and Gamburd, see \cite{BumpGamburd}, of the integrand in terms of the Schur polynomials $s_{\lambda}(\mathbf{z})$:
\begin{align*}
\mathbb{E}\left(\prod_{j=1}^{k} |P_N\left(U,\theta_j\right)|^{2\beta}\right)=\frac{s_{(N,\dots,N,0,\dots,0)}\left(e^{\i \theta_1},\dots,e^{\i \theta_1},e^{\i \theta_2},\dots, e^{\i \theta_2},\dots, e^{\i \theta_k},\dots, e^{\i \theta_k}\right)}{\prod_{j=1}^{k}e^{\i N \beta \theta_j} }
\end{align*}
where each of the strings of $N$'s and $0$'s in the signature is of length $k\beta$ and each variable $e^{\i \theta_l}$ for $1\le l \le k$ appears $2\beta$ times.

 Then, using the well-known combinatorial formula (sum over tableaux) for Schur polynomials indexed by $\lambda\in \mathsf{S}_M^+$:
\begin{align*}
s_{\lambda}(z_1,\dots,z_M)=\sum_{t \in \mathsf{SSYT}(\lambda)}^{}z_1^{\textnormal{number of } 1's \textnormal{ in } t} z_2^{\textnormal{number of } 2's \textnormal{ in } t} \cdots z_M^{\textnormal{number of } M's \textnormal{ in } t}
\end{align*}
in the integral we obtain the result.
\end{proof}

We have the following equivalent representation for $\textnormal{MoM}_N\left(k,\beta\right)$ in terms of Gelfand-Tsetlin patterns:

\begin{prop}\label{GelfandProp}
Let $k,\beta \in \mathbb{N}$. $\textnormal{MoM}_N\left(k,\beta\right)$ is equal to the number of Gelfand-Tsetlin patterns of depth $2k\beta$ and top row $\lambda^{(2k\beta)}= \left(N,\dots,N,0,\dots,0\right)$ where each of the strings of $N$'s and $0$'s is of length $k\beta$, and which moreover satisfy the following constraints
\begin{align}\label{GelfandTsetlinConstraint}
\lambda_1^{(2j\beta)}+\cdots+\lambda_{2j\beta}^{(2j\beta)}=Nj\beta, \ \textnormal{ for } j \in \{1, \dots, k\}.
\end{align}
We denote the set of such Gelfand-Tsetlin patterns by $\mathcal{GT}\left(N;k;\beta\right)$.
\end{prop}

\begin{proof}
We apply the bijection $\mathfrak{B}$. It suffices to observe that the constraints (\ref{SemistandardConstraint}) at the tableaux level get mapped to the constraints (\ref{GelfandTsetlinConstraint}) at the Gelfand-Tsetlin pattern level and vice-versa
\end{proof}

\begin{rmk}
Observe that for $j=k$ the constraint (\ref{GelfandTsetlinConstraint}) on the sum is superfluous (as it clearly was for the constraint (\ref{SemistandardConstraint}) at the tableaux level). 
\end{rmk}

\subsection{Integer array representation}

We now make a final translation of the problem, based on a simple observation that is however the crux of the argument. We need a couple of definitions.

\begin{defn}\label{InterlacingSquare}
Let $\mathfrak{I}_m\subset{\mathbb{R}}^{m^2}$ denote the collection of real $m\times m$ matrices whose entries are non-decreasing along rows and non-increasing along columns. More precisely, for $(x_i^{(j)})_{i,j=1}^m\in \mathfrak{I}_m$ (here and throughout the subscript denotes the column and the superscript in parentheses the row) we have: 
\begin{align*}
x_1^{(j)} \le x_2^{(j)}\le \cdots  \le x_{m}^{(j)}, \ \textnormal{ for } j \in \{1, \dots,m\}.\\
x_i^{(1)} \ge x_i^{(2)}\ge \cdots  \ge x_{i}^{(m)}, \ \textnormal{ for } i \in \{1, \dots,m\}.
\end{align*}
\end{defn}

We define a certain set of integer arrays with constraints:

\begin{defn}
Let $\mathcal{M}_N(k,\beta)$ be the set of integer arrays $x=\left(x_i^{(j)}\right)_{i,j=1}^{k\beta} \in \mathbb{Z}^{(k\beta)^2}$ such that:
\begin{itemize}
\item[(1)] $0\le x_i^{(j)} \le N$, for $1\le i,j \le k\beta$,
\item[(2)] for $l=1, \dots, \left\lfloor\frac{k}{2} \right\rfloor$:
\begin{align*}
x_1^{(2\beta l)}+x_2^{(2\beta l-1)}+\cdots +x_{2\beta l}^{(1)}&=l\beta N,\\
x_{k\beta-2\beta l+1}^{(k\beta)}+x_{k\beta-2\beta l+2}^{(k\beta-1)}+\cdots+ x_{k\beta}^{(k\beta-2\beta l+1)}&=l\beta N,
\end{align*}
\item[(3)] the matrix $\left(x_i^{(j)}\right)_{i,j=1}^{k\beta}$ is in the set $\mathfrak{I}_{k\beta}$.
\end{itemize}
Note that, there is a total number of $k-1$ constraints in $(2)$ above (in case $k$ is even, the two constraints for $l=\lfloor\frac{k}{2} \rfloor=\frac{k}{2}$ coincide).
\end{defn}

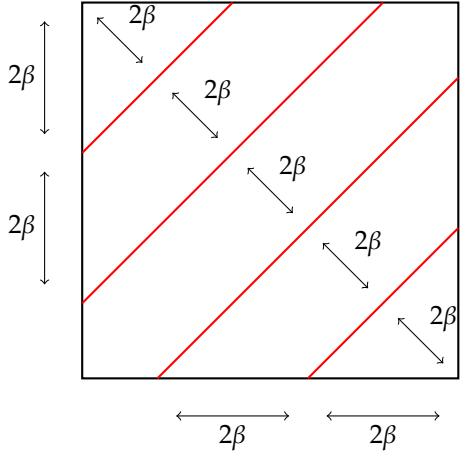
\begin{figure}
\begin{tikzpicture}

\draw [thick] (0,0) rectangle (5,5);
\draw [thick,red] (0,3) -- (2,5);
\draw [thick,red] (0,1) -- (4,5);
\draw [thick,red] (1,0) -- (5,4);
\draw [thick,red] (3,0) -- (5,2);
\draw [<->] (-0.5,4.75) -- (-0.5,3.25);
\node[left] at (-0.5,4) {$2\beta$};
\draw [<->] (-0.5,2.75) -- (-0.5,1.25);
\node[left] at (-0.5,2) {$2\beta$};
\draw [<->] (1.25,-0.5) -- (2.75,-0.5);
\node[below] at (2,-0.5) {$2\beta$};
\draw [<->] (3.25,-0.5) -- (4.75,-0.5);
\node[below] at (4,-0.5) {$2\beta$};

\draw[<->] (0.2,4.8) -- (0.8,4.2);
\node[above right] at (0.5,4.5) {$2\beta$};

\draw[<->] (1.2,3.8) -- (1.8,3.2);
\node[above right] at (1.5,3.5) {$2\beta$};

\draw[<->] (2.2,2.8) -- (2.8,2.2);
\node[above right] at (2.5,2.5) {$2\beta$};

\draw[<->] (3.2,1.8) -- (3.8,1.2);
\node[above right] at (3.5,1.5) {$2\beta$};

\draw[<->] (4.2,0.8) -- (4.8,0.2);
\node[above right] at (4.5,0.5) {$2\beta$};

\end{tikzpicture}

\caption{An element of $\mathcal{M}_N(k,\beta)$, with $k=5$. The labels $2\beta$ denote the number of rows/columns/diagonals along the $\longleftrightarrow$ arrows. The sum constraints in Part (2) of the definition of $\mathcal{M}_N(k,\beta)$ are on the red diagonals. There is a total of $k-1=4$ such constraints.}\label{Figure1}
\end{figure}

We now observe that there is a natural bijection, which we denote by $\mathfrak{S}_{(N;k;\beta)}$, between $\mathcal{GT}\left(N;k;\beta\right)$ and $\mathcal{M}_N\left(k,\beta\right)$. This can be seen as follows. It is an important observation that the form of $\lambda^{(2k\beta)}=(N,\dots,N,0,\dots,0)$ introduces a large number of constraints, see Figure \ref*{Figure2} for an illustration. Let us first look at $\lambda^{(2k\beta-1)}$. By the interlacing $\lambda^{(2k\beta-1)}\prec \lambda^{(2k\beta)}$, we get that there is only one free coordinate in $\lambda^{(2k\beta-1)}$:
\begin{align*}
\lambda_1^{(2k\beta-1)}, \dots , \lambda_{k\beta-1}^{(2k\beta-1)}\equiv N,\\
\lambda_{k\beta+1}^{(2k\beta-1)}, \dots , \lambda_{2k\beta-1}^{(2k\beta-1)}\equiv 0,\\
0 \le \lambda^{(2k\beta-1)}_{k\beta} \le N.
\end{align*}
We relabel $x_1^{(1)}=\lambda_{k\beta}^{(2k\beta-1)}$. Looking at $\lambda^{(2k\beta-2)}$, again due to the interlacing $\lambda^{(2k\beta-2)}\prec \lambda^{(2k\beta-1)}$, we see that there are only two free coordinates:
\begin{align*}
\lambda_1^{(2k\beta-2)},\dots, \lambda_{k\beta-2}^{(2k\beta-2)}\equiv N,\\
\lambda_{k\beta+1}^{(2k\beta-2)},\dots, \lambda_{2k\beta-2}^{(2k\beta-2)}\equiv 0,\\
0\le \lambda_{k\beta-1}^{(2k\beta-2)},\lambda_{k\beta}^{(2k\beta-2)}\le N,
\end{align*}
which moreover satisfy:
\begin{align*}
\lambda_{k\beta}^{(2k\beta-2)} \le \lambda_{k\beta}^{(2k\beta-1)}=x_1^{(1)}\le \lambda_{k\beta-1}^{(2k\beta-2)}.
\end{align*}
We relabel them as follows:
\begin{align*}
x_1^{(2)}=\lambda_{k\beta}^{(2k\beta-2)}, \ x_2^{(1)}=\lambda_{k\beta-1}^{(2k\beta-2)}.
\end{align*}
Continuing in this fashion, by relabelling the non-fixed coordinates of the signatures $\{\lambda^{(i)}\}_{i=1}^{2k\beta}$, we obtain the desired bijection $\mathfrak{S}_{(N;k;\beta)}$, which we formalize shortly. Observe that, under this relabelling the interlacing constraints in $\mathcal{GT}(N;k;\beta)$ exactly correspond to the constraints in Part (3) of Definition \ref{InterlacingSquare}.  We also note that the constraints $(\ref{GelfandTsetlinConstraint})$ are easily seen to correspond to the sum constraints in the definition of $\mathcal{M}_N(k,\beta)$. Finally, observe that after $k\beta$ steps of this process (of relabelling starting from the top signature), none of the coordinates are necessarily fixed to be $0$ or $N$.

\begin{figure}

 \begin{tikzpicture}
   
   \node[red] at (0,10) {$\bf{0}$};
   \node[red] at (1,10) {$\bf{0}$};
   \node[red] at (2,10) {$\cdots$};
     \node[red] at (3,10) {$\bf{0}$};
       \node[red] at (4,10) {$\bf{N}$};
    \node[red] at (5,10) {$\cdots$};
  \node[red] at (6,10) {$\bf{N}$};
     \node[red] at (7,10) {$\bf{N}$};
        \node[red] at (0.5,9) {$\bf{0}$};
        \node[red] at (1.5,9) {$\cdots$};
          \node[red] at (2.5,9) {$\bf{0}$};
        \node[] at (3.5,9) {$\lambda_{k\beta}^{(2k\beta-1)}$};
            \node[red] at (4.5,9) {$\bf{N}$};
         \node[red] at (5.5,9) {$\cdots$};
       \node[red] at (6.5,9) {$\bf{N}$};
       
           \node[red] at (1,8) {$\bf{0}$};
                           \node[red] at (1.75,8) {$\cdots$};
                           \node[red] at (2.25,8) {$\bf{0}$};
               \node[] at (3,8) {$\lambda_{k\beta}^{(2k\beta-2)}$};
               \node[] at (4,8) {$\lambda_{k\beta-1}^{(2k\beta-2)}$};
                   \node[red] at (4.75,8) {$\bf{N}$};
                   
                \node[red] at (5.5,8) {$\cdots$};
              \node[red] at (6,8) {$\bf{N}$};
              
               \node[] at (2,7) {$\vdots$};
              \node[] at (3.25,7) {$\vdots$};
              \node[] at (4.5,7) {$\vdots$};
              
             \node[red] at (1.25,6) {$\bf{0}$};
             \node[] at (2.25,6) {$\lambda_{k\beta}^{(k\beta+1)}$};
              \node[] at (3.25,6) {$\cdots$};
            \node[] at (4.5,6) {$\lambda_{2}^{(k\beta+1)}$};
            \node[red] at (5.5,6) {$\bf{N}$};

            \node[] at (1.5,5) {$\lambda_{k\beta}^{(k\beta)}$};
             \node[] at (2.5,5) {$\lambda_{k\beta-1}^{(k\beta)}$};
              \node[] at (3.25,5) {$\cdots$};
          \node[] at (4,5) {$\lambda_{2}^{(k\beta)}$};    
           \node[] at (5,5) {$\lambda_{1}^{(k\beta)}$};  
              
             \node[] at (2,4) {$\lambda_{k\beta-1}^{(k\beta-1)}$};
               \node[] at (3.25,4) {$\cdots$};
             \node[] at (4.5,4) {$\lambda_{1}^{(k\beta-1)}$};  
             
              \node[] at (3.25,3) {$\vdots$};
                            \node[] at (2.5,3) {$\vdots$};
                            \node[] at (4,3) {$\vdots$};
                            
                          \node[] at (2.75,2) {$\lambda_2^{(2)}$};
                          \node[] at (3.75,2) {$\lambda_1^{(2)}$};
            \node[] at (3.25,1) {$\lambda_1^{(1)}$};
   \end{tikzpicture}

\caption{An element of $\mathcal{GT}(N;k;\beta)$. Observe that, there is a number of fixed coordinates, due to the interlacing, at $0$ and $N$ (for any element in $\mathcal{GT}(N;k;\beta)$). Moreover, every $2\beta$ signatures there is a sum constraint of the form (\ref{GelfandTsetlinConstraint}).}\label{Figure2}
\end{figure}
Formally, the bijection $\mathfrak{S}_{(N;k;\beta)}$ (essentially a relabelling) is given as follows:
\begin{align*}
\mathfrak{S}_{(N;k;\beta)}:\mathcal{GT}\left(N;k;\beta\right) &\longrightarrow \mathcal{M}_N\left(k,\beta\right)\\
\lambda_{k\beta-i+1}^{(2k\beta-j)}&\mapsto x_i^{(j-i+1)}, \ j=1,\dots,k\beta-1;i=1,\dots,j;\\
\lambda_{k\beta+2-j-i}^{(k\beta+1-j)}&\mapsto x_{i+j-1}^{(k\beta-i+1)}, j=1,\dots,k\beta; i=1,\dots, k\beta+1-j.
\end{align*}
It is important to note that, as illustrated in Figure \ref*{Figure2}, for $j=1,\dots,k\beta-1$ we have:
\begin{align*}
\lambda_1^{(2k\beta-j)},\dots,\lambda_{k\beta-j}^{(2k\beta-j)}\equiv N,\\
\lambda_{k\beta+1}^{(2k\beta-j)},\dots,\lambda_{2k\beta-j}^{(2k\beta-j)}\equiv 0.
\end{align*}
It is immediate, by simply writing the sums out, that the constraints (\ref{GelfandTsetlinConstraint}) exactly correspond to the sum constraint in Part (2) of the definition of $\mathcal{M}_N(k,\beta)$.

Pictorially, see Figure \ref*{Figure3}, the bijection $\mathfrak{S}_{(N;k;\beta)}$ between $\mathcal{GT}\left(N;k;\beta\right)$ and $\mathcal{M}_N\left(k,\beta\right)$ goes as follows: the non-fixed coordinates (from top to bottom) of a Gelfand-Tsetlin pattern in $\mathcal{GT}\left(N;k;\beta\right)$ are obtained by reading a unique array in $\mathcal{M}_N\left(k,\beta\right)$ sequentially along each diagonal (in the north-east direction), from the top-left to the bottom-right corner and vice versa (as indicated by the blue arrow in Figure \ref*{Figure3}).

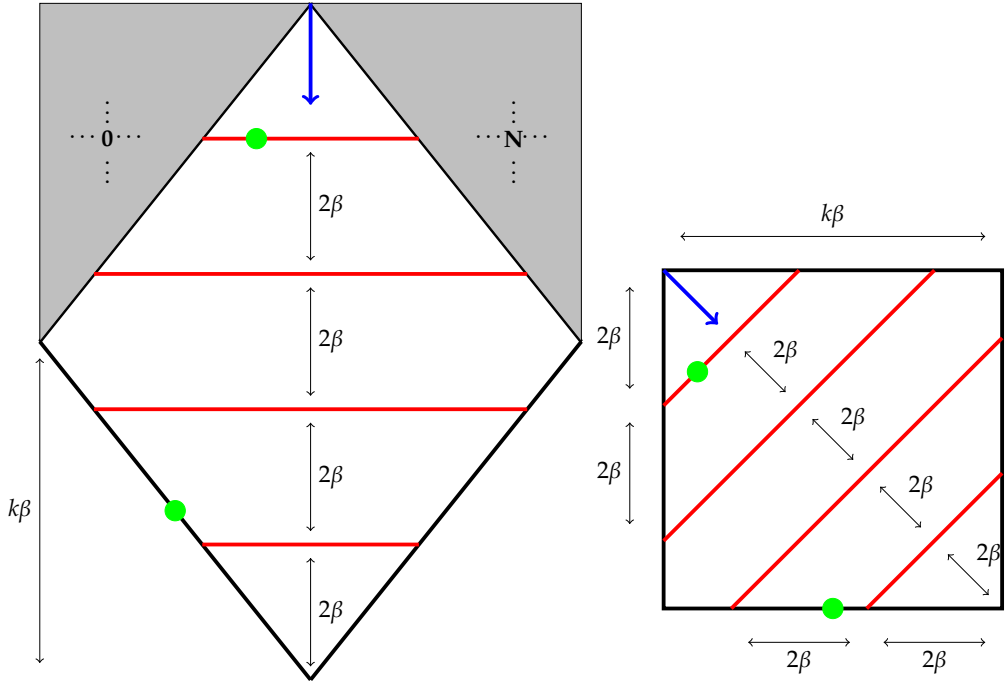
\begin{figure}

\scalebox{0.9}{
\begin{tikzpicture}

\draw[ultra thick]   (5,0) -- (1,5);
\draw[ultra thick]   (5,10) -- (1,5);
\draw[ultra thick]   (5,10) -- (9,5);
 \draw[ultra thick]   (5,0) -- (9,5);

 \draw [fill=lightgray] (1,5) -- (1,10)-- (5,10)--(1,5); 
 
  \draw [fill=lightgray] (9,5) -- (9,10)-- (5,10)--(9,5); 
  \node[] at (2,8) {$\mathbf{0}$};
  \node[above] at (2,8.1) {$\vdots$};
    \node[below] at (2,8) {$\vdots$};
    \node[left] at (2,8) {$\cdots$};
        \node[right] at (2,8) {$\cdots$};
        
     \node[] at (8,8) {$\mathbf{N}$};
     \node[above] at (8,8.1) {$\vdots$};
       \node[below] at (8,8) {$\vdots$};
       \node[left] at (8,8) {$\cdots$};
           \node[right] at (8,8) {$\cdots$};     
           
   \draw[ultra thick,red]   (3.4,8) -- (6.6,8);         
     \draw[fill,green] (4.2,8) circle [radius=0.15];   
       \draw[ultra thick,red]   (3.4,2) -- (6.6,2);         
            \draw[fill,green] (3,2.5) circle [radius=0.15];   
      \draw[ultra thick,red]   (1.8,4) -- (8.2,4);            
      
\draw[ultra thick,red]   (1.8,6) -- (8.2,6);            

\draw[ultra thick, blue,->] (5,10)-- (5,8.5);

\draw [<->] (5,1.8) -- (5,0.2);
\node[right] at (5,1){$2\beta$};
\draw [<->] (5,2.2) -- (5,3.8);
\node[right] at (5,3){$2\beta$};
\draw [<->] (5,4.2) -- (5,5.8);
\node[right] at (5,5){$2\beta$};
\draw [<->] (5,6.2) -- (5,7.8);
\node[right] at (5,7){$2\beta$};
\draw [<->] (1,0.25) -- (1,4.75);
\node[left] at (1,2.5){$k\beta$};
\end{tikzpicture}
\begin{tikzpicture}

\draw [ultra thick] (0,0) rectangle (5,5);
\draw [ultra thick,red] (0,3) -- (2,5);
\draw[fill,green] (0.5,3.5) circle [radius=0.15];   
\draw[fill,green] (2.5,0) circle [radius=0.15];   
\draw [ultra thick,red] (0,1) -- (4,5);
\draw [ultra thick,red] (1,0) -- (5,4);
\draw [ultra thick,red] (3,0) -- (5,2);
\draw [<->] (-0.5,4.75) -- (-0.5,3.25);
\node[left] at (-0.5,4) {$2\beta$};
\draw [<->] (-0.5,2.75) -- (-0.5,1.25);
\node[left] at (-0.5,2) {$2\beta$};
\draw [<->] (1.25,-0.5) -- (2.75,-0.5);
\node[below] at (2,-0.5) {$2\beta$};
\draw [<->] (3.25,-0.5) -- (4.75,-0.5);
\node[below] at (4,-0.5) {$2\beta$};

\draw [<->] (0.25,5.5) -- (4.75,5.5);
\node[above] at (2.5,5.5) {$k\beta$};

\draw[->,ultra thick, blue] (0,5) -- (0.8,4.2);

\draw[<->] (1.2,3.8) -- (1.8,3.2);
\node[above right] at (1.5,3.5) {$2\beta$};

\draw[<->] (2.2,2.8) -- (2.8,2.2);
\node[above right] at (2.5,2.5) {$2\beta$};

\draw[<->] (3.2,1.8) -- (3.8,1.2);
\node[above right] at (3.5,1.5) {$2\beta$};

\draw[<->] (4.2,0.8) -- (4.8,0.2);
\node[above right] at (4.5,0.5) {$2\beta$};

\end{tikzpicture}
}

\caption{An element of $\mathcal{GT}(N;k;\beta)$ and its image, under the relabelling, $\mathfrak{S}_{(N;k;\beta)}$ in $\mathcal{M}_N(k,\beta)$. The shaded regions correspond to the fixed coordinates at $0$ and $N$. The red horizontal lines (signatures) correspond to the constraints $(\ref{GelfandTsetlinConstraint})$ and get mapped under $\mathfrak{S}_{(N;k;\beta)}$ to the diagonal red lines of the element in $\mathcal{M}_N(k,\beta)$. Also indicated is how the two green points (coordinates $\lambda_i^{(j)}$ for some $i,j$) get mapped under $\mathfrak{S}_{(N;k;\beta)}$}\label{Figure3}

\end{figure}

Thus, by applying $\mathfrak{S}_{(N;k;\beta)}$ and using Proposition \ref{GelfandProp} we obtain:
\begin{prop}\label{LatticePointRepProp}
Let $k,\beta \in \mathbb{N}$. $\textnormal{MoM}_N\left(k,\beta\right)=\#\mathcal{M}_N\left(k,\beta\right)$.
\end{prop}

\section{Lattice point count asymptotics}
We first record the following classical theorem on the number of lattice points in convex regions of Euclidean space, see for example Section 2 in \cite{LatticePointCount}.
\begin{thm}\label{LatticePointCountTheorem}
Assume $S \subset \mathbb{R}^{l}$ is a convex region contained in a closed ball of radius $\rho$. Then,
\begin{align}
\#\left(S \cap \mathbb{Z}^l \right)=\textnormal{vol}_l\left(S\right)+O_l\left(\rho^{l-1}\right),
\end{align}
where the implicit constant in the error term depends only on $l$.
\end{thm}

Our aim is to apply this theorem. To this end, consider the index set:
\begin{align*}
\mathcal{S}_{(k,\beta)}=\bigg\{ (i,j), \ 1 \le i,j \le k\beta: (i,j) \neq (1,2\beta l), (k\beta-2\beta l+1,k\beta ), \textnormal{ for } l=1,\dots, \left\lfloor \frac{k}{2}\right\rfloor \bigg \}.
\end{align*}
Observe that $\mathcal{S}_{(k,\beta)}$ has $(k\beta)^2-(k-1)$ elements. Now consider the region, denoted by $\mathcal{V}_{(k,\beta)}$, contained in $\mathbb{R}^{(k\beta)^2-(k-1)}$ and defined by the following system of inequalities:
\begin{itemize}
\item[(1)] $0\le u_i^{(j)} \le 1$, for all $(i,j) \in S_{(k,\beta)}$,
\item[(2)] for $l=1, \dots, \left\lfloor\frac{k}{2} \right\rfloor$:
\begin{align*}
u_1^{(2\beta l)}&=l\beta-u_2^{(2\beta l-1)}-\cdots -u_{2\beta l}^{(1)}, \textnormal{ we have } 0\le u_1^{(2\beta l)} \le 1, \\
u_{k\beta-2\beta l+1}^{(k\beta)}&=l\beta -u_{k\beta-2\beta l+2}^{(k\beta-1)}-\cdots -u_{k\beta}^{(k\beta-2\beta l+1)}, \textnormal{ we have } 0 \le u_{k\beta-2\beta l+1}^{(k\beta)} \le 1,
\end{align*}
\item[(3)] the matrix $\left(u_i^{(j)}\right)_{i,j=1}^{k\beta}$ is in the set $\mathfrak{I}_{k\beta}$.
\end{itemize}

Note that, $\mathcal{V}_{(k,\beta)}$ is convex as it is an intersection of convex sets (half planes). Moreover, $\mathcal{V}_{(k,\beta)}$ is contained in $[0,1]^{(k\beta)^2-(k-1)}$ and therefore contained in a closed ball of radius $\sqrt{(k\beta)^2-(k-1)}$.

We are now in a position to prove Theorem \ref{MainTheorem}.

\begin{proof}[Proof of Theorem \ref{MainTheorem}]
Observe that, 
\begin{align}
\# \mathcal{M}_N\left(k,\beta\right)=\# \left(\mathbb{Z}^{k^2\beta^2-(k-1)}\cap \left(N \mathcal{V}_{(k,\beta)}\right)\right).
\end{align}
Here, $N \mathcal{V}_{(k,\beta)}=\{N x: x\in  \mathcal{V}_{(k,\beta)}\}$ is the dilate of $\mathcal{V}_{(k,\beta)}$ by a factor of $N$.

Thus, from Proposition \ref{LatticePointRepProp} and Theorem \ref{LatticePointCountTheorem}  with $S=N \mathcal{V}_{(k,\beta)}$ (whose conditions are satisfied as shown above) we obtain:
\begin{align}
\textnormal{MoM}_N\left(k,\beta\right)=\# \mathcal{M}_N\left(k,\beta\right)=\# \left(\mathbb{Z}^{k^2\beta^2-(k-1)}\cap \left(N \mathcal{V}_{(k,\beta)}\right)\right)=\textnormal{vol}\left(N\mathcal{V}_{(k,\beta)}\right)+O_{k,\beta}\left(N^{(k\beta)^2-k}\right).
\end{align}
Since, 
\begin{align*}
\textnormal{vol}\left(N\mathcal{V}_{(k,\beta)}\right)=N^{(k\beta)^2-(k-1)}\textnormal{vol}\left(\mathcal{V}_{(k,\beta)}\right)
\end{align*}
the statement of the theorem follows with $\mathfrak{c}(k,\beta)=\textnormal{vol}\left(\mathcal{V}_{(k,\beta)}\right)$.
\end{proof}

\section{On the leading order coefficient $\mathfrak{c}(k,\beta)$}

\subsection{An expression involving volumes of trapezoidal continuous Gelfand-Tsetlin patterns}
The aim of this section is to obtain an explicit expression for the leading order coefficient $\mathfrak{c}(k,\beta)=\textnormal{vol}\left(\mathcal{V}_{(k,\beta)}\right)$ in terms of (integrals of determinants of) $B$-splines, certain piecewise polynomial functions on $\mathbb{R}$, introduced by Curry and Schoenberg \cite{CurrySchoenberg}. These objects appear in approximation theory and the study of total positivity, \cite{Karlin}. We begin with some preliminary definitions.

We define the continuous Weyl chamber $W^n$ by (note the reversal in the order of coordinates compared to the definition of partitions/signatures, this is to keep with the notations of \cite{OlshanskiProjections} from which we draw some of our formulae):
\begin{align*}
W^n=\big\{x=(x_1,\dots,x_n)\in \mathbb{R}^n:x_1\le \cdots \le x_n \big\}.
\end{align*}
We say that $y \in W^n$ and $x \in W^{n+1}$ interlace and still write $y \prec x$ if:
\begin{align*}
x_1\le y_1\le x_2 \le \dots \le x_n \le y_n\le x_{n+1}.
\end{align*}
We also write $W^n_{[0,1]}$ for the Weyl chamber with coordinates in $[0,1]$:
\begin{align*}
W_{[0,1]}^n=\big\{x=(x_1,\dots,x_n)\in [0,1]^n:x_1\le \cdots \le x_n \big\}.
\end{align*}
The definition of a continuous Gelfand-Tsetlin pattern (also referred to as a Gelfand-Tsetlin polytope) is completely analogous to the discrete setting: an interlacing sequence $\{x^{(i)} \}_{i=1}^n$ in $\{W^i\}_{i=1}^n$,
\begin{align*}
x^{(1)}\prec x^{(2)}\prec \cdots \prec x^{(n-1)}\prec x^{(n)}.
\end{align*}
We also call an interlacing sequence $\{x^{(i)} \}_{i=m}^n$ in $\{W^i\}_{i=m}^n$ with $n-m\ge 1$,
\begin{align*}
x^{(m)}\prec x^{(m+1)}\prec \cdots \prec x^{(n-1)}\prec x^{(n)},
\end{align*}
a 'trapezoidal' continuous Gelfand-Tsetlin pattern.

For $x \in W^n$ and $a \in W^m$, with $n-m\ge 2$ and $m\ge 1$, we define $\mathsf{Vol}_{m}^{n}(x,a)$ as follows (the volume of a 'trapezoidal' continuous Gelfand-Tsetlin pattern with fixed top row $x \in W^n$ and fixed bottom row $a \in W^m$):
\begin{align*}
\mathsf{Vol}_{m}^{n}(x,a)=\int_{W^{m+1}\times \cdots \times W^{n-1}}dy^{(m+1)}\cdots dy^{(n-1)}\mathbf{1}\left(a\prec y^{(m+1)}\right)\mathbf{1}\left(y^{(m+1)}\prec y^{(m+2)}\right)\cdots \mathbf{1}\left(y^{(n-1)}\prec x\right).
\end{align*}
Observe that, if $x \in W^n$ and $a \in W^m$ do not belong to an interlacing sequence then $\mathsf{Vol}_{m}^{n}(x,a)\equiv 0$.
For $x \in W^n$, with $n\ge 2$, we define $\mathsf{Vol}^{n}(x)$ by (the volume of a standard continuous Gelfand-Tsetlin pattern with fixed top row $x \in W^n$):
\begin{align*}
\mathsf{Vol}^{n}(x)=\int_{W^{1}\times \cdots \times W^{n-1}}dy^{(1)}\cdots dy^{(n-1)}\mathbf{1}\left(y^{(1)}\prec y^{(2)}\right)\mathbf{1}\left(y^{(2)}\prec y^{(3)}\right)\cdots \mathbf{1}\left(y^{(n-1)}\prec x\right).
\end{align*}
For notational convenience we define a final quantity that will appear in our formulae for $\textnormal{vol}\left(\mathcal{V}_{(k,\beta)}\right)$ when $k$ is even. We thus define $\mathsf{Vol}_{(m,n)}(x,y)$, for $n-m\ge 1$ and $x,y \in W^m$, as follows:
\begin{align*}
\mathsf{Vol}_{(m,n)}(x,y)&\overset{\textnormal{def}}{=}\int_{W^{m+1}\times \cdots \times W^n \times W^{n-1} \times \cdots \times W^{m+1}}dz^{(m+1)}\cdots dz^{(n)}d\tilde{z}^{(n-1)}\cdots d\tilde{z}^{(m+1)}\mathbf{1}\left(x\prec z^{(m+1)}\right)\times\cdots \\
&\ \ \ \ \ \  \ \ \ \cdots \times \mathbf{1}\left(z^{(n-1)}\prec z^{(n)}\right)\mathbf{1}\left(\tilde{z}^{(n-1)}\prec z^{(n)}\right)\cdots \mathbf{1}\left(y\prec \tilde{z}^{(m+1)}\right)\\
&=\int_{W^n}^{}\mathsf{Vol}_{m}^{n}(z,x)\mathsf{Vol}_{m}^{n}(z,y)dz.
\end{align*} 
We can now give an explicit expression for $\mathfrak{c}(k,\beta)=\textnormal{vol}\left(\mathcal{V}_{(k,\beta)}\right)$ in terms of the quantities above.
\begin{prop}Let $\beta \in \mathbb{N}$. Then, if $k$ is even:
\begin{align*}
\mathfrak{c}(k,\beta)&=\textnormal{vol}\left(\mathcal{V}_{(k,\beta)}\right)\\&=\int_{W^{2\beta}_{[0,1]}\times \cdots \times W^{k\beta}_{[0,1]}\times W^{(k-1)\beta}_{[0,1]}\times \cdots \times W^{2\beta}_{[0,1]}}dx^{(1)}\cdots dx^{\left(\frac{k}{2}\right)}d\tilde{x}^{\left(\frac{k}{2}-1\right)}\cdots d\tilde{x}^{(1)}\mathsf{Vol}^{2\beta}(x^{(1)})\mathsf{Vol}_{2\beta}^{4\beta}(x^{(2)},x^{(1)})\cdots\\
\cdots &\times\mathsf{Vol}_{(k-2)\beta}^{k\beta}\left(x^{\left(\frac{k}{2}\right)},x^{\left(\frac{k}{2}-1\right)}\right) \mathsf{Vol}_{(k-2)\beta}^{k\beta}\left(x^{\left(\frac{k}{2}\right)},\tilde{x}^{\left(\frac{k}{2}-1\right)}\right)\times \mathsf{Vol}_{(k-4)\beta}^{(k-2)\beta}\left(\tilde{x}^{\left(\frac{k}{2}-1\right)},\tilde{x}^{\left(\frac{k}{2}-2\right)}\right)\cdots \mathsf{Vol}^{2\beta}(\tilde{x}^{(1)})\\ &\times\prod_{j=1}^{\frac{k}{2}}\delta\left(\sum_{i=1}^{2\beta j}x_{i}^{(j)}-\beta j \right)\prod_{j=1}^{\frac{k}{2}-1}\delta\left(\sum_{i=1}^{2\beta j}\tilde{x}_{i}^{(j)}-\beta j\right).
\end{align*}
While, if $k$ is odd:
\begin{align*}
\mathfrak{c}(k,\beta)&=\textnormal{vol}\left(\mathcal{V}_{(k,\beta)}\right)\\&=\int_{W^{2\beta}_{[0,1]}\times \cdots \times W^{(k-1)\beta}_{[0,1]}\times W^{(k-1)\beta}_{[0,1]}\times \cdots \times W^{2\beta}_{[0,1]}}dx^{(1)}\cdots dx^{\left(\frac{k-1}{2}\right)}d\tilde{x}^{\left(\frac{k-1}{2}\right)}\cdots d\tilde{x}^{(1)}\mathsf{Vol}^{2\beta}(x^{(1)})\mathsf{Vol}^{4\beta}_{2\beta}(x^{(2)},x^{(1)})\cdots\\
\cdots \times& \mathsf{Vol}_{(k-3)\beta}^{(k-1)\beta}\left(x^{\left(\frac{k-1}{2}\right)},x^{\left(\frac{k-1}{2}-1\right)}\right)\mathsf{Vol}_{((k-1)\beta,k\beta)}\left(x^{\left(\frac{k-1}{2}\right)},\tilde{x}^{\left(\frac{k-1}{2}\right)}\right) \mathsf{Vol}_{(k-3)\beta}^{(k-1)\beta}\left(\tilde{x}^{\left(\frac{k-1}{2}\right)},\tilde{x}^{\left(\frac{k-1}{2}-1\right)}\right)\times \cdots \times \mathsf{Vol}^{2\beta}(\tilde{x}^{(1)})\\
&\times\prod_{j=1}^{\frac{k-1}{2}}\delta\left(\sum_{i=1}^{2\beta j}x_{i}^{(j)}-\beta j \right)\prod_{j=1}^{\frac{k-1}{2}}\delta\left(\sum_{i=1}^{2\beta j}\tilde{x}_{i}^{(j)}-\beta j \right).
\end{align*}
Here, $\delta(\cdot)$ is the Dirac delta-function.
\end{prop}

\begin{proof}
We first observe that $\textnormal{vol}\left(\mathcal{V}_{(k,\beta)}\right)$ is equal to the following integral (each $z^{(i)}$ and $\tilde{z}^{(j)}$ corresponds to a diagonal of an element in $\mathcal{V}_{(k,\beta)}$ written as a matrix):
\begin{align*}
\int dz^{(1)}\cdots dz^{(k\beta-1)}dz^{(k\beta)}d\tilde{z}^{(k\beta-1)}\cdots d\tilde{z}^{(1)} \mathbf{1}\left(z^{(1)}\prec z^{(2)}\right)\cdots \mathbf{1}\left(z^{(k\beta-1)}\prec z^{(k\beta)}\right)\mathbf{1}\left(\tilde{z}^{(k\beta-1)}\prec z^{(k\beta)}\right)\cdots\times \\
\cdots \times \mathbf{1}\left(\tilde{z}^{(1)}\prec \tilde{z}^{(2)}\right)\prod_{j=1}^{\lfloor \frac{k}{2}\rfloor}\delta \left(\sum_{i}^{2\beta j}z_i^{(2\beta j)}-j\beta\right)\prod_{j=1}^{\lfloor \frac{k}{2}\rfloor}\delta \left(\sum_{i}^{2\beta j}\tilde{z}_i^{(2\beta j)}-j\beta\right),
\end{align*}
where the integral is over the region:
\begin{align*}
(z^{(1)},\dots,z^{(k\beta-1)},z^{(k\beta)},\tilde{z}^{(k\beta-1)},\dots,\tilde{z}^{(1)})\in W^1_{[0,1]}\times \cdots \times W^{k\beta-1}_{[0,1]}\times  W^{k\beta}_{[0,1]}\times W^{k\beta-1}_{[0,1]}\times \cdots \times W^1_{[0,1]}
\end{align*}
and if $k$ is even we write $\tilde{z}^{(k\beta)}=z^{(k\beta)}$ (as remarked before, in this case one of the two sum constraints is superfluous).

The statement of the proposition follows after we perform the integrations between two consecutive sum constraints, which are given by the quantities $\mathsf{Vol}_m^n, \mathsf{Vol}^n, \mathsf{Vol}_{(m,n)}$ by their definition.
\end{proof}

\begin{rmk}
Observe that, we can restrict the integrations above over the interiors of the Weyl chambers $\mathring{W}^n_{[0,1]}$.
\end{rmk}

\subsection{Continuous Gelfand-Tsetlin pattern volumes and B-splines}

We now turn our attention to $\mathsf{Vol}_{m}^{n}(z,y)$ which is the basic building block in the expression for $\textnormal{vol}\left(\mathcal{V}_{(k,\beta)}\right)$. Firstly, the simpler quantity $\mathsf{Vol}^n(x)$ has a well-known explicit expression, see for example Corollary 3.2 of \cite{OlshanskiProjections}, given by
\begin{align}
\mathsf{Vol}^n(x)=\frac{\Delta_n(x)}{\prod_{j=1}^{n}(j-1)!},
\end{align}
where for $x \in W^n$ we denote by $\Delta_n(x)$ the Vandermonde determinant
\begin{align*}
\Delta_n(x)=\prod_{1\le i <j\le n}^{}(x_j-x_i).
\end{align*}
Now, for $y \in \mathring{W}^n$, the interior of the chamber $W^n$, we define the $B$-spline $a\mapsto M(a;y_1,\dots,y_n)$ with knots $y_1<\cdots<y_n$ by the explicit expression:
\begin{align*}
M(a;y_1,\dots,y_n)=(n-1) \sum_{i:y_i>a}\frac{(y_i-a)^{n-2}}{\prod_{r:r\neq i}^{}(y_i-y_r)}.
\end{align*}
Equivalently, see \cite{OlshanskiProjections}, \cite{CurrySchoenberg}, it can be defined as the only function $a\mapsto M(a;y_1,\dots,y_n)$ on $\mathbb{R}$ of class $C^{n-3}$, vanishing outside the interval $(y_1,y_n)$, equal to a polynomial of degree $\le n-2$ on each interval $(y_i,y_{i+1})$ and normalized by the condition:
\begin{align*}
\int_{-\infty}^{+\infty}M(a;y_1,\dots,y_n)da=1.
\end{align*}
An explicit expression for $\mathsf{Vol}_{m}^{n}(z,y)$, for $z \in \mathring{W}^{n}$, is due to Olshanski \cite{OlshanskiProjections}. (Also, observe that when $z \in \partial W^n$, the boundary of the chamber, namely when some of the coordinates coincide, $\mathsf{Vol}_{m}^{n}(z,y)\equiv 0$ and similarly $\mathsf{Vol}^{n}(z)\equiv 0$.) This follows from the proof of Theorem 3.3 in \cite{OlshanskiProjections} (see Step 2 of the proof of that theorem and display (15) therein) and is given by
\begin{align}
\mathsf{Vol}_{m}^{n}(z,y)=\mathsf{const}_{m}^n\prod_{1\le j-i \le n-m}^{}(z_j-z_i)\det \left[M(y_j;z_i,\dots,z_{n-m+i})\right]_{i,j=1}^m
\end{align}
where $\mathsf{const}_{m}^n$ is explicit:
\begin{align*}
\mathsf{const}_{m}^n=\frac{1}{\left(\left(n-m\right)!\right)^m}\prod_{j=1}^{n-m}\frac{1}{(j-1)!}.
\end{align*}

\subsection{The case $k=2$ and Painleve V}
Unlike the argument for the asymptotics of $\textnormal{MoM}_N(k,\beta)$ which is generic for any $k\in \mathbb{N}$, the study of the constant $\mathfrak{c}(k,\beta)$ appears to be quite different depending on whether $k=2$ or $k\ge3$ (when $k=1$ of course $\textnormal{MoM}_N(1,\beta)$ is completely explicit).

For $k=2$, the expression for $\textnormal{vol}\left(\mathcal{V}_{(2,\beta)}\right)$ only involves $\mathsf{Vol}^{2\beta}$ which is simply a Vandermonde determinant multiplied by a constant. Putting everything together we obtain:
\begin{align*}
\mathfrak{c}(2,\beta)=\textnormal{vol}\left(\mathcal{V}_{(2,\beta)}\right)=\frac{1}{(2\beta)!G(1+2\beta)^2}\int_{[0,1]^{2\beta}}\delta\left(\sum_{i=1}^{2\beta}t_i-\beta\right)\prod_{1\le i<j \le 2\beta}^{}(t_i-t_j)^2dt_1\cdots dt_{2\beta},
\end{align*}
where $G(1+2\beta)=1!2!3!\cdots(2\beta-1)!$ is the Barnes G-function.

Then, by writing the Dirac delta-function as
\begin{align*}
\delta(c)=\int_{-\infty}^{\infty}\exp\left(2\pi \i c y\right)dy
\end{align*}
and using the Andreif identity, it is possible to show that, see Section 2 of \cite{BasorPainleveV} for the details (in their notation $k=2\beta,c=\beta$):
\begin{align*}
\mathfrak{c}(2,\beta)=\frac{1}{G(1+2\beta)^2}\int_{-\infty}^{\infty}\exp\left(2\pi \i \beta u\right)D_{2\beta}\left(2\pi \i u\right)du,
\end{align*}
where the Hankel determinant $D_{2\beta}$ is given by:
\begin{align*}
D_{2\beta}(t)&=\det \left[g^{(i+j-2)}(t)\right]_{i,j=1}^{2\beta},\\
g^{(n)}(t)&=\int_{0}^{1}(-x)^n\exp(-tx)dx.
\end{align*}
Finally, $D_{2\beta}(t)$ is known to fall into a special class of Hankel determinants which have a representation in terms of the Painleve V equation, see \cite{BasorGeRubinstein},\cite{BasorPainleveV}. More precisely let
\begin{align*}
H_{2\beta}(t)=t\frac{d}{dt}\log D_{2\beta}(t)+(2\beta)^2.
\end{align*}
Then, we have:
\begin{align*}
\left(t\frac{d^2}{dt^2}H_{2\beta}(t)\right)^2=\left(H_{2\beta}(t)+\left(4\beta-t\right)\frac{d}{dt}H_{2\beta}(t)\right)^2-4\left(\frac{d}{dt}H_{2\beta}(t)\right)^2\left((2\beta)^2-H_{2\beta}(t)+t\frac{d}{dt}H_{2\beta}(t)\right).
\end{align*}

As soon as one moves to $k=3$ the expression for $\textnormal{vol}(\mathcal{V}_{3,\beta})$ involves the quantity $\mathsf{Vol}_{(2\beta,3\beta)}$ in the integrand, which is given in terms of the $B$-splines and which is more complicated to analyse. In particular, it is a very interesting problem to understand whether there exists a useful Hankel determinant representation for $\mathfrak{c}(3,\beta)$.

\bigskip
\noindent
{\sc Mathematical Institute, University of Oxford, Oxford, OX2 6GG, UK.}\newline
\href{mailto:theo.assiotis@maths.ox.ac.uk}{\small theo.assiotis@maths.ox.ac.uk}

\bigskip
\noindent
{\sc Mathematical Institute, University of Oxford, Oxford, OX2 6GG, UK.}\newline
\href{mailto:jon.keating@maths.ox.ac.uk}{\small jon.keating@maths.ox.ac.uk}
\end{document}